\documentclass{birkmult}

\usepackage{amsmath}
\usepackage{amssymb}
\usepackage{tensor}

\renewcommand{\subjclassname}{\textup{2000} Mathematics Subject Classification}

\newtheorem{theorem}{Theorem}[section]

\newtheorem{corollary}[theorem]{Corollary}

\theoremstyle{definition}

\newtheorem{example}[theorem]{Example}

\theoremstyle{remark}
\newtheorem{remark}[theorem]{Remark}

\numberwithin{equation}{section}

 \def\R{{\mathbb R}}

 \def\C{{\mathbb C}}

 \def\cL{{\mathcal L}}

 \def\cL{{\mathcal L}}

 \def\mr{{\rm MR}_p (X, D_A, D_B)}

 \def\mr1{{\rm MR}_{p,1} (X, D_A, D_B)}

 \newcommand{\ud}{\mathrm{d}}

\newcommand{\dom}[1]{\mathsf{D}_{#1}}     

\begin{document}
%
%
%
%
%
%
%
%
%

\title[Maximal regularity in interpolation spaces]{Maximal regularity in interpolation spaces for second order Cauchy problems}

\author{Charles J. K. Batty }
\address{St. John's College, University of Oxford, Oxford OX1 3JP, United Kingdom}
\email{charles.batty@sjc.ox.ac.uk}

\author{Ralph Chill}
\address{Technische Universit\"at Dresden,
Institut f\"ur Analysis,
01062 Dresden,
Germany}
\email{ralph.chill@tu-dresden.de}

\author{Sachi Srivastava}

\address{Department of Mathematics,
University of Delhi South Campus,
Benito Juarez Road,
New Delhi 110021, India}
\email{sachi\_srivastava@yahoo.com}


\date{\today}

\renewcommand{\subjclassname}{\textup{2000} Mathematics Subject Classification}






\date{\today}


\begin{abstract}
We study maximal regularity in interpolation spaces for the sum of three closed linear operators on a Banach space, and we apply the abstract results to obtain Besov and H\"older maximal regularity for complete second order Cauchy problems under natural parabolicity assumptions. We discuss applications to partial differential equations.
\end{abstract}
 
\renewcommand{\subjclassname}{\textup{2000} Mathematics Subject Classification}

\maketitle

\section{Introduction}
We study  maximal regularity results in certain time interpolation spaces for the second order Cauchy problem
\begin{equation} \label{second}
\begin{split}
& \ddot u + B \dot u + A u = f  \quad \text{in } [0,T] , \\
& u(0) = u_0 , \,\, \dot u (0) = u_1 .
\end{split} 
\end{equation}
Here $A$ and $B$ are closed linear operators defined on a complex Banach space $X$ with domains  $\dom{A}$ and $\dom{B}$ respectively. By a maximal regularity result we mean a result which asserts that for every $f$ in a certain function space $\mathbb{E} \subseteq L^1 (0,T;X)$ and homogeneous initial values $u_0 = u_1 = 0$ the problem \eqref{second} admits a unique strong solution $u$ satisfying $\ddot u$, $B\dot u$, $Au\in \mathbb{E}$. In particular, the three terms on the left-hand side of \eqref{second} have the same regularity as the given right-hand side.\\

The notion of $L^{p}$ maximal regularity (that is, $\mathbb{E} = L^p (0,T;X)$) and, closely connected with it, of maximal regularity in rearrangement invariant Banach function spaces for this abstract second order Cauchy problem has been studied in Chill \& Srivastava \cite{ChSr05, ChSr08} and Chill \& Kr\'ol \cite{ChKr14}.  See also Arendt et al. \cite{ACFP07}, Batty, Chill \& Srivastava \cite{BaChSr08}, Cannarsa, Da Prato \& Zol\'esio \cite{CaDPZo90}, Dautray \& J.-L. Lions \cite[Chapter XVIII, Section 5]{DaLi87VIII}, Favini \cite{Fa91} and Yakoubov \cite{Ya09} for generalisations to non-autonomous problems, Bu \& Fang \cite{BuFa08a,BuFa10}, Keyantuo \& Lizama \cite{KeLi06} for second order problems with periodic boundary conditions, Fern{\'a}ndez, Lizama \& Poblete \cite{FeLiPo10} for third order problems, Zacher \cite{Za05} for Volterra equations and Bu \cite{Bu10,Bu11b}, Keyantuo \& Lizama \cite{KeLi11}, Lizama \& Poblete \cite{LiPb08} for fractional order problems with periodic boundary conditions.\\

The notion of maximal regularity for the second order problem generalises in a natural way the notion of maximal regularity for the first order problem   
\begin{equation} \label{first}
\begin{split}
& \dot u + A u = f \quad \text{in } [0,T] , \\ 
& u(0) = 0 ,
\end{split} 
\end{equation} 
which in turn goes back to the notion of maximal regularity of the sum of two closed linear operators on a Banach space by Da Prato \& Grisvard \cite{DPGr75}; see also the monograph by Lunardi \cite[Theorem 3.18]{Lu09}. In particular, Da Prato and Grisvard showed that if $A$, $D$ are two sectorial operators with domains $\dom{A}$ and $\dom{D}$ and sectoriality angles $\varphi_A$ and $\varphi_D$, respectively, and if $\varphi_A + \varphi_D <\pi$, then for every $x$ in a real interpolation space between $X$ and $\dom{D}$ (or $\dom{A}$) there is a unique solution $y$ of the operator equation $ Ay + D y = x$ lying in the space $ \dom{A} \cap \dom{D} $ with $ Ay$ and $Dy $ belonging to the same interpolation space. This result was then applied to the first order Cauchy problem (\ref{first}) by taking $D$ to be the differentiation operator on, for example, $L^p (0,T;X)$ or $C([0,T];X)$. The real interpolation spaces then include the Besov spaces and the H\"older spaces, respectively, that is, one obtains Besov or 
H\"older maximal regularity, the latter also being called optimal regularity in the literature.\\ 
 
Analogously, maximal regularity of the sum of three operators corresponds with the definition of maximal regularity of the second order problem (\ref{second}) as mentioned above.   In this paper we follow the idea of \cite{DPGr75} to prove a maximal regularity result on interpolation spaces for the second order abstract problem (\ref{second}) for $u_0=u_1=0$.  Here no assumptions are needed on the space $X$, and the operators $A$ and $B$ are not required to satisfy assumptions about functional calculus or $R$-boundedness.  Moreover, we extend the result in such a way that in principle we can also treat initial value problems of the general form (\ref{second}), although here the identification of the associated trace spaces remains an open problem.  The conclusions provide Besov or H\"older maximal regularity of second order Cauchy problems; see also Favini et al. \cite{CFLMM12,FLMM13,FLMTY08}, Mezeghrani \cite{Mz13} for similar results for elliptic problems with inhomogeneous Dirichlet boundary conditions or 
Bu \cite{Bu09,Bu10}, Bu \& Fang \cite{BuFa10}, Keyantuo \& Lizama \cite{KeLi06,KeLi06a}, Poblete \cite{Pb07,Pb09} for second order problems with periodic boundary conditions or on the line. \\
 
The paper is organised as follows. Section \ref{sec0} is of a preliminary nature where we recall relevant definitions and facts. In Section \ref{sec1} an abstract result concerning maximal regularity of certain sums of three closed operators is proven. As an application we obtain our main result on the maximal regularity for the second order Cauchy problem in Section \ref{sec2}. Section \ref{sec.3} is devoted to the initial value problem, while examples of applications can be found in Section \ref{sec.4}.

\section{Preliminaries} \label{sec0}

Let $X$ be a complex Banach space. Whenever $(D,\dom{D})$ is a closed linear operator on $X$, $\theta\in (0,1)$, $1\leq p\leq \infty$, we denote by 
\begin{align*}
 \mathsf{D}_D (\theta ,p) & := (X,\dom{D} )_{\theta ,p} \quad \text{ and} \\
\mathsf{D}_D (\theta ) & := (X,\dom{D} )_{\theta} 
\end{align*}
the real interpolation spaces between $X$ and $\dom{D}$ (the latter space is equipped with the graph norm), as defined by the $K$-method or the trace method. Recall that $\mathsf{D}_D (\theta )$ is a closed subspace of $\mathsf{D}_D (\theta ,\infty )$.

The operator $D$ is called {\em sectorial of angle} $\varphi \in (0,\pi )$ if 
\[
\sigma (D) \subseteq \Sigma_{\varphi} := \{ \lambda\in\C : |{\rm arg}\, \lambda |\leq \varphi \} \cup \{ 0\} ,
\]
and for every $\varphi' \in (\varphi , \pi )$ one has
\[
\sup_{\lambda\not\in \Sigma_{\varphi'}} \| \lambda R(\lambda ,D)\| <\infty .
\]
For a sectorial operator $D$, for $\theta\in (0,1)$ and $1\leq p\leq \infty$ we have the equalities 
\begin{align*}
\mathsf{D}_D (\theta , p) & = \{ x\in X : \| t^\theta D(t+D)^{-1} x\| \in L^p (0,\infty ; \frac{\ud t}{t} ) \} \text{ and} \\
\mathsf{D}_D (\theta ) & = \{ x\in \mathsf{D}_D (\theta, \infty ) : \lim_{t\to\infty} t^{\theta} D(t+D)^{-1} x = 0 \} ,
\end{align*}
and 
\[
\| x\|_{\theta, p} :=  \| x \| +  \| t^\theta D(t+D)^{-1} x\|_{L^p (0,\infty ; \frac{\ud t}{t} )}  
\]
is an equivalent norm on the interpolation space $\mathsf{D}_D (\theta ,p)$ \cite[Propositions 2.2.2 and 2.2.6]{Lu95}.  Note that if $\varphi' \in (\varphi,\pi)$, then $-e^{\pm i\varphi' }D$ is sectorial and its domain is $\dom{D}$ with an equivalent graph norm.  Hence
\begin{equation} \label{lpang}
\mathsf{D}_D (\theta , p) = \{ x\in X : \| t^\theta D R(te^{\pm i \varphi'},D)^{-1} x\| \in L^p (0,\infty ; \frac{\ud t}{t} ).
\end{equation}

\begin{example} \label{d.1}
Let $X$ be a Banach space, and fix $1\leq p < \infty$. If $D_{max}$ is the differentiation operator on $L^p (0,T; X)$ with maximal domain, that is,
\begin{align*}
\dom{D_{max}} & := W^{1,p} (0,T;X) , \\
 D_{max} u & := \dot u  ,
\end{align*}
then we have
\begin{align*}
\mathsf{D}_{D_{max}} (\theta ,q)  = B^{\theta}_{pq} (0,T ; X) \qquad (\theta\in (0,1),\, 1\leq q < \infty ) ,
\end{align*}
and, in particular,
\[
\mathsf{D}_{D_{max}} (\theta ,p) = W^{\theta ,p} (0,T ;X )  \qquad (\theta\in (0,1)) ; 
\]
compare with \cite[Exercise 6, p.18]{Lu09} or \cite[Theorem, p.204]{Tr83}. Here, $B^\theta_{pq}$ and $W^{\theta ,p}$ are the Besov spaces and fractional order Sobolev spaces defined respectively by
\begin{align*}
B^{\theta}_{pq} (0,T ; X) & := \{ u\in L^p (0,T;X) : \int_0^T \!\! \left( \int_0^T \frac{\| u(t) - u(s)\|^p}{|t-s|^{\theta p}} \; ds \right)^{q/p} \!\!\!\!\! dt < \infty \} , \\
W^{\theta ,p} (0,T ;X ) & := \{ u\in L^p (0,T;X) : \int_0^T \int_0^T \frac{\| u(t) - u(s)\|^p}{|t-s|^{\theta p}} \; ds \; dt < \infty \} .
\end{align*}
If $D$ is the restriction of the differentiation operator $D_{max}$ on $L^p (0,T;X)$ to the domain
\begin{align*}
 \dom{D} & = \mathring{W}^{1,p} (0,T;X) := \{ u\in W^{1,p} (0,T;X) : u(0) = 0 \} ,
\end{align*}
then 
\[
 \mathsf{D}_D (\theta ,q) = \begin{cases}
                             \mathring{B}^{\theta}_{pq} (0,T;X) & \text{ if } \theta \geq \frac{1}{p} , \\
                             B^{\theta}_{pq} (0,T ; X) & \text{ if } \theta < \frac{1}{p} ;
                            \end{cases}
\]
compare with \cite[Theorem, p.210]{Tr83}, where actually two-sided homogeneous boundary conditions were considered. Here $\mathring{B}^{\theta}_{pq}$ is the space of all functions $u\in B^{\theta}_{pq}$ with trace $u(0)=0$; note that the trace is well defined whenever $\theta \geq\frac{1}{p}$, since then the Besov space $B^{s}_{pq}$ is embedded into the space of continuous functions. On the other hand, 
\[
 \mathsf{D}_{D_{max}} (\theta ,q) = \mathsf{D}_{D} (\theta ,q) \text{ whenever } \theta <\frac{1}{p} .
\]
Recall that the operator $D$ is sectorial of angle $\frac{\pi}{2}$, while $D_{max}$ is not sectorial (in fact, $\sigma (D_{max}) = \C$). 
\end{example}

\begin{example} \label{d.2}
If $D_{max}$ is the differentiation operator on $C ([0,T] ;X)$ 
with maximal domain, that is,
\begin{align*}
\dom{D_{max}} & := C^1 ([0,T] ; X) , \\
 D_{max} u & := \dot u  ,
\end{align*}
then we have
\begin{align*}
\mathsf{D}_{D_{max}} (\theta ,\infty ) & = C^{\theta} ([0,T] ; X) \text{ and} \\
\mathsf{D}_{D_{max}} (\theta ) & = h^\theta ([0,T] ;X ) ; 
\end{align*}
compare with \cite[Example 1.9 and Exercise 5, p.18]{Lu09}. Here, $C^\theta$ and $h^\theta$ are the H\"older and little H\"older spaces defined respectively by
\begin{align*}
C^\theta ([0,T] ; X) & := \{ u\in C([0,T] ; X) : \sup_{t,s\in [0,T] \atop s\not= t} \frac{\| u(t) - u(s)\|}{|t-s|^\theta} <\infty \} \text{ and} \\
h^\theta ([0,T] ; X) & := \{ u\in C^\theta ([0,T];X) : \lim_{|t-s|\to 0} \frac{\| u(t) - u(s)\|}{|t-s|^\theta} = 0 \} .
\end{align*}
If $D$ is the restriction of the differentiation operator $D_{max}$ to the domain 
\[
 \dom{D}  := \mathring{C}^1 ([0,1] ; X)  := \{ u\in C^1 ([0,T] ;X) : u(0) = 0\} 
\]
(so that $D$ is no longer densely defined), then 
\begin{align*}
 \mathsf{D}_D (\theta ,\infty ) & = \mathring{C}^\theta ([0,T];X) \\
& := \{ u\in C^{\theta} ([0,T];X ) : u(0)=0 \} \text{ and} \\
 \mathsf{D}_D (\theta ) & = \mathring{h}^\theta ([0,T];X) \\
& := \{ u\in h^{\theta} ([0,T];X ) : u(0)=0 \} .
\end{align*}
Also in this example, $D$ is sectorial of angle $\frac{\pi}{2}$ while $\sigma (D_{max}) = \C$.
\end{example}

\section{An abstract theorem} \label{sec1}

Our main abstract result is a maximal regularity result for the sum of three closed, linear, commuting operators.  We say that an operator $A$ {\it commutes} with an invertible operator $D$ if $AD \subseteq DA$, or equivalently if $D^{-1}A \subseteq AD^{-1}$.  Here the compositions such as $AD$ have their natural domains.

\begin{theorem} \label{main}
Let $A$, $B$ and $D$ be three closed, linear operators on a Banach space $X$ with both $ A, B $  commuting with $D$. Assume that
\begin{itemize}
\item[(a)] the operator $D$ is invertible and sectorial of angle $\varphi_1\in (0,\pi )$,
\item[(b)] there exists $\varphi_2 \in (\varphi_1 ,{\pi} )$, such that $H(\lambda ):= (\lambda^2 + \lambda B + A)^{-1}$ exists in $\cL (X)$ for every $\lambda\in\Sigma_{\varphi_2}$,
\item[(c)] $H\in H^\infty (\Sigma_{\varphi_2 } ; \cL (X))$, and
\item[(d)] the functions  
\begin{align*}
\lambda & \mapsto \lambda^2 \, H (\lambda) , \\
\lambda & \mapsto \lambda\, B \, H (\lambda) , \mbox{ and} \\
\lambda & \mapsto A \, H (\lambda)
\end{align*}
are uniformly bounded in $\Sigma_{\varphi_2}$ with values in $\cL (X)$.
\end{itemize}
Then, for every $\theta\in (0,1)$ and every $1\leq p\leq\infty$ the operator $L_{\theta,p}$ on $\mathsf{D}_D (\theta ,p)$ given by
\begin{align*}
\dom{L_{\theta ,p}} & := \{ x \in \dom{D^2} \cap \dom{BD} \cap \dom{A} : D^2 x, \, BD x, \, Ax \in \mathsf{D}_{D} (\theta , p) \} , \\
L_{\theta ,p} x & := D^2 x + BD x + Ax ,
\end{align*}
is closed and boundedly invertible. More precisely, if we define
\begin{equation} \label{S}
Sx := \frac{1}{2\pi i} \int_\Gamma R(\lambda , D) H(\lambda )x \; \ud\lambda , \quad x\in X ,
\end{equation}
where $\Gamma$ is a path connecting $e^{i\varphi'}\infty$ with $e^{-i\varphi'} \infty$ for some $\varphi'\in (\varphi_1 , \varphi_2 )$ and surrounding $\sigma (D)$, then $S\in \cL (X)$, $S$ is a left-inverse of $D^2+BD +A$ in $X$, and for every $x\in \mathsf{D}_D (\theta , p)$ one has
\begin{align*}
& Sx\in \dom{D^2} \cap \dom{BD} \cap \dom{A} , \\
& D^2 Sx , \, BD Sx , \, A Sx \in \mathsf{D}_D (\theta ,p ) \text{ and} \\
& L_{\theta ,p} Sx = (D^2 + BD + A)Sx = x ,
\end{align*}
that is, $S$ restricted to $\mathsf{D}_D (\theta ,p)$ is the bounded inverse of $L_{\theta ,p}$. A similar result holds for $\mathsf{D}_D (\theta )$ instead of  $\mathsf{D}_D (\theta ,p)$.
\end{theorem}

\begin{proof}
It follows from the assumptions, more precisely from the estimates on $R(\lambda ,D)$ and $H$, that the integral in \eqref{S} converges absolutely and that $S$ is thus a bounded operator on $X$. Note that since $ 0 \notin \sigma(D),$ the path $ \Gamma $ may be chosen so that $ 0 \notin \Gamma$ and lying in the sector $\Sigma_{\varphi_2}$.  Since $A$ and $B$ commute with $D$, the bounded operators $H(\lambda)$ and $R(\lambda,D)$ commute with each other.

Let us first prove that $S$ is a left-inverse of $D^2 +BD+A$. By definition of $S$, for every $x\in \dom{D^2}\cap \dom{BD} \cap \dom{A}$ we can calculate, using the Lebesgue's dominated convergence theorem,
\begin{align*}
\lefteqn{S (D^2 + BD + A)x } \\
& = \frac{1}{2\pi i} \int_\Gamma R(\lambda ,D) H(\lambda ) (D^2 + BD + A) x \; \ud\lambda \\
& = \lim_{s\to +\infty} \frac{1}{2\pi i} \int_\Gamma \frac{s}{s+\lambda} R(\lambda ,D) H(\lambda ) (D^2 - \lambda^2 + B (D-\lambda ) + \lambda^2 + \lambda B + A) x \; \ud\lambda \\
& = - \lim_{s\to +\infty} \frac{1}{2\pi i} \int_\Gamma \frac{s}{s+\lambda} H(\lambda ) (D+\lambda + B)x \; \ud\lambda + \lim_{s\to +\infty} \frac{1}{2\pi i} \int_\Gamma \frac{s}{s+\lambda} R(\lambda ,D) x \; \ud\lambda  \\
& = x .
\end{align*}
In the last step we have used the identities  
\[
\lim_{s\to +\infty} \frac{1}{2\pi i} \int_\Gamma \frac{s}{s+\lambda} R(\lambda ,D) x \; \ud\lambda = x
\]
and 
\[
\frac{1}{2\pi i} \int_\Gamma \frac{s}{s+\lambda} H(\lambda ) (D+\lambda + B)x \; \ud\lambda = 0  \text{ for every } s>0 ,
\] 
which follow from a simple application of Cauchy's residue theorem, remembering that $x \in  \dom{D^2}\cap \dom{BD}$ and the estimate on the function $H$ in assumption (d); compare also with \cite[Proposition 2.1.4 (i)]{Lu95}. 
 
Now, fix $\theta\in (0,1)$ and $1\leq p\leq \infty$. We show that $S$ is a right inverse of $L_{\theta ,p}$ in $\mathsf{D}_D (\theta ,p)$. By definition of $S$, for every $t>0$ and every $x\in X$ we have, by the resolvent identity,
\begin{equation} \label{eq.DS}
\begin{split}
(t+D)^{-1} Sx & = \frac{1}{2\pi i} \int_\Gamma \frac{1}{t+\lambda} (t+D)^{-1} H(\lambda )x \; \ud\lambda + \\
& \phantom{+} + \frac{1}{2\pi i} \int_\Gamma \frac{1}{t+\lambda} R(\lambda ,D) H(\lambda )x \; \ud\lambda \\
& = \frac{1}{2\pi i} \int_\Gamma \frac{1}{t+\lambda} R(\lambda ,D) H(\lambda )x \; \ud\lambda \\
& = \frac{1}{2\pi i} \int_\Gamma \frac{1}{t+\lambda} H(\lambda ) R(\lambda ,D)x \; \ud\lambda.
\end{split}
\end{equation}
For every $t >0$ and every $x\in\mathsf{D}_D (\theta ,p)$, we have
\begin{align*}
g_1(r) &:= \frac{(re^{\pm i \varphi'})^{1-\theta}}{t+re^{\pm i \varphi'}} \in L^q(0,\infty; \frac{dr}{r}), \qquad 1 \le q \le \infty, \\
g_2(r) &:= \|r^\theta R(re^{\pm i \varphi'}, D)x\| \in L^p(0,\infty; \frac{dr}{r}),
\end{align*}
by \eqref{lpang}, and $\|AH(re^{\pm i \varphi'})\|$ is bounded by assumption (d).  By H\"older's inequality,  the integral
\begin{equation*}
\frac{1}{2\pi i} \int_\Gamma \frac{1}{t+\lambda} AH(\lambda )  \, DR(\lambda ,D) x \; \ud\lambda = \frac{1}{2\pi i} \int_\Gamma \frac{\lambda^{1-\theta}}{t+\lambda} AH(\lambda ) \, \lambda^\theta DR(\lambda ,D)x \; \frac{\ud\lambda}{\lambda}
\end{equation*}
converges absolutely. Since $A$ is closed,  we conclude from this and \eqref{eq.DS} that, for every $x\in \mathsf{D}_D (\theta ,p)$,  $y:=D(t+D)^{-1}Sx \in \dom{A}$.  Since $A$ commutes with $D$, $Sx = tD^{-1}y + y \in \dom{A}$ and 
\[
D(t+D)^{-1} ASx = AD(t+D)^{-1}Sx = \frac{1}{2\pi i} \int_\Gamma \frac{1}{t+\lambda} AH(\lambda ) DR(\lambda ,D) x \; \ud\lambda. 
\]
Let
\[
g_3(r) = \frac{r^\theta}{|\cos \varphi'|+r} \in L^1(o,\infty; \frac{dr}{r}).
\]
By assumption (d), we can estimate 
\begin{align*}
\|t^\theta D(t+D)^{-1} ASx\| & \leq \frac{C}{\pi} \int_0^\infty g_3(t/r) g_2(r) \; \frac{\ud r}{r} ,
\end{align*}
It follows from Young's inequality (applied to the multiplicative group $(0,\infty)$ with the Haar measure $\frac{1}{t}\ud t$) that 
\[
\| t^\theta D(t+D)^{-1} ASx \| \in L^p (0,\infty ; \frac{\ud t}{t} )
\]
as well. This proves that $ASx\in D_D (\theta ,p)$.

Similarly, we deduce that for every $x\in \mathsf{D}_D (\theta ,p)$ one has $ Sx \in \dom{BD} $ and 
\begin{align*}
D(t+D)^{-1} BDSx & = \lim_{s\to +\infty} \frac{1}{2\pi i} \int_\Gamma \frac{s}{s+\lambda} \frac{1}{t+\lambda} BD H(\lambda ) \, DR(\lambda ,D) x \; \ud\lambda \\
& = \lim_{s\to +\infty} \frac{1}{2\pi i} \int_\Gamma \frac{s}{s+\lambda} \frac{1}{t+\lambda} \, \lambda B H(\lambda ) \, DR(\lambda ,D)x \; \ud\lambda \\
& \phantom{=} - \lim_{s\to +\infty} \frac{1}{2\pi i} \int_\Gamma \frac{s}{s+\lambda} \frac{1}{t+\lambda} BD H(\lambda ) x \; \ud\lambda \\
& = \int_\Gamma \frac{1}{t+\lambda} \, \lambda B H(\lambda ) \, DR(\lambda ,D)x \; \ud\lambda  .
\end{align*}
This computation is also justified by \eqref{eq.DS} (for the first equality), by Cauchy's theorem and since the integral on the right-hand side converges absolutely for every $x\in \mathsf{D}_D (\theta, p)$, using again assumption (d). One can proceed similarly as above and one obtains $BD Sx\in \mathsf{D}_D (\theta ,p)$. Similar arguments prove that $ Sx \in \dom{D^{2}} $ and  $D^{2}Sx \in \mathsf{D}_D (\theta ,p).$  Indeed, by Cauchy's theorem,
\begin{align*}
D(t + D )^{-1} D^{2} Sx  &= \lim_{s\to +\infty} \frac{1}{2\pi i} \int_\Gamma \frac{s}{s+\lambda} \frac{1}{t+\lambda} D^{2}H(\lambda ) D R(\lambda ,D) x \; \ud\lambda \\ 
& = \lim_{s\to +\infty} \frac{1}{2\pi i} \int_\Gamma \frac{s}{s+\lambda} \frac{1}{t+\lambda} D^{2}H(\lambda )  x \; \ud\lambda
+ \\
& ~~~~ +\lim_{s\to +\infty} \frac{1}{2\pi i} \int_\Gamma \frac{s}{s+\lambda} \frac{1}{t+\lambda}\lambda  D H(\lambda ) x \; \ud\lambda\\
& ~~~~- \lim_{s\to +\infty} \frac{1}{2\pi i} \int_\Gamma \frac{s}{s+\lambda} \frac{1}{t+\lambda} \lambda ^{2} D H(\lambda ) R(\lambda ,D) x \; \ud\lambda \\
& = \frac{1}{2\pi i} \int_\Gamma \frac{1}{t+\lambda} \lambda ^{2} H(\lambda ) \, DR(\lambda ,D) x \; \ud\lambda.
\end{align*} 
Assumption (d) allows us to proceed as in the previous case.
  
From what we have proved above, it follows that the operator $S$ leaves $\mathsf{D}_D (\theta ,p)$ invariant. By the closed graph theorem, the restriction of $S$ to $\mathsf{D}_D (\theta ,p)$ is bounded. Moreover, the above equalities show that $ SL_{\theta ,p} x = x $ for all $ x \in \dom{L_{\theta ,p}} $ and $ L_{\theta ,p} S y = y $ for all $ y \in \mathsf{D}_{D}(\theta, p ).$ Thus $ L_{\theta ,p} : \dom{L_{\theta ,p}} \to \mathsf{D}_{D}(\theta, p ) $ is boundedly invertible (and necessarily closed) with $ L_{\theta ,p}^{-1} = S|_{\mathsf{D}_D (\theta ,p)}$.     
\end{proof}

\section{Maximal regularity of the second order Cauchy problem in interpolation spaces} \label{sec2}

In this section we consider the second order Cauchy problem with homogeneous initial data:
\begin{equation} \label{second.hom}
\begin{split}
& \ddot u + B \dot u + A u = f  \quad \text{in } [0,T] , \\
& u(0) = \dot u (0) = 0.
\end{split} 
\end{equation} 
 
\begin{theorem} \label{main2}
Let $A$ and $B$ be two closed linear operators on a Banach space $X$. Assume that
\begin{itemize}
\item[(b)] there exists $\varphi \in (\frac{\pi}{2} ,{\pi} )$, such that $H(\lambda ):= (\lambda^2 + \lambda B + A)^{-1}$ exists in $\cL (X)$ for every $\lambda\in\Sigma_{\varphi}$,
\item[(c)] $H\in H^\infty (\Sigma_{\varphi } ; \cL (X))$, and
\item[(d)] the functions  
\begin{align*}
\lambda & \mapsto \lambda^2 \, H (\lambda) , \\
\lambda & \mapsto \lambda\, B \, H (\lambda) , \mbox{ and} \\
\lambda & \mapsto A \, H (\lambda)
\end{align*}
are uniformly bounded in $\Sigma_{\varphi}$ with values in $\cL (X)$.
\end{itemize}
Then the problem \eqref{second.hom} has Besov and H\"older regularity in the following sense : for every $f\in B^\theta_{pq} (0,T;X),  \theta < \frac{1}{p}$
(resp. $f\in \mathring{B}^{\theta}_{pq} (0,T;X)$ if $\theta \geq \frac{1}{p}$, $f\in \mathring{C}^\theta ([0,T];X)$, $f\in \mathring{h}^\theta ([0,T];X)$) the problem \eqref{second.hom} admits a unique solution $u$ satisfying 
\begin{multline*}
u, \, \dot u , \, \ddot u , \, B\dot u , \, Au \in B^\theta_{pq} (0,T;X) \quad \\
(\text{resp. } \in \mathring{B}^{\theta}_{pq} (0,T;X),\, \mathring{C}^\theta ([0,T];X),\, \mathring{h}^\theta ([0,T];X) ) .
\end{multline*}
\end{theorem}

\begin{proof}
Let the operators $\bar{A}$, $\bar{B}$ and $D$ defined on the space $L^p (0,T;X)$ (resp. $C ([0,T] ;X)$) be given as follows:
\begin{align*}
(\bar{A} u) (t) & := A u(t) , \\
(\bar{B} u) (t) & := B u(t) , \\
(Du) (t) & := \dot u(t) ,
\end{align*}
Here the multiplication operators $\bar{A}$ and $\bar{B}$ have their natural domains, that is, 
\[
\dom{\bar{A}} = L^p(0,T; \dom{A} )
\] 
and similarly for $ \dom{\bar{B}} $, and 
\begin{align*}
\dom{D} & := \mathring{W}^{1,p} (0,T;X) \\
\text{(resp. } \dom{D} & := \mathring{C}^1 ([0,T];X) \,\, \text{)} ; 
\end{align*}
compare with Examples \ref{d.1} and \ref{d.2}. Recall that  $D$ is sectorial of angle $\frac{\pi}{2}$. Applying Theorem \ref{main} to the above operators and noting that (compare with Examples \ref{d.1} and \ref{d.2})
\begin{align*}
\mathsf{D}_{D} (\theta, p ) & = B^{\theta}_{p q}( 0, T; X) \quad \text{if } \theta <\frac{1}{p} , \\
(\text{resp. } \mathsf{D}_{D} (\theta, p ) & = \mathring{B}^{\theta}_{p q}( 0, T; X) \quad \text{if } \theta \geq \frac{1}{p} , \\
\mathsf{D}_{D}(\theta, p) & = \mathring{C}^{\theta}([0,T];X) , \\
\mathsf{D}_{D} (\theta ) & = \mathring{h}^\theta ([0,T] ;X) \,\, ) , 
\end{align*}
we obtain the required maximal regularity.  
\end{proof}

The fractional power $A^\varepsilon$ in the following corollary is defined by using any standard functional calculus for sectorial operators; see, for example, \cite{Hs06}.

\begin{corollary} \label{epsilon.1}
Consider the abstract second order Cauchy problem 
\begin{equation} \label{eq.1}
\begin{array}{ll}
\ddot{u}(t) + \alpha A^{\varepsilon}  \dot{u}(t)+ Au(t) = f(t), & t \in [0, T],\\[2mm]
u(0) = \dot{u}(0) = 0,
\end{array}
\end{equation}
where $ A $ is a sectorial operator of angle $\varphi\in (0,\pi )$ on a Banach space $X,$ and $ \varepsilon > 0$, $\alpha > 0.$  Assume that one of the following conditions holds: 
\begin{itemize}
\item[(a)] $\varepsilon = \frac{1}{2}$, $\alpha \geq 2$, and $\varphi\in (0,\pi)$.
\item[(b)] $\varepsilon = \frac{1}{2}$, $\alpha\in (0,2)$, and $\varphi\in (0, \pi - 2 \arctan \frac{\sqrt{4-\alpha^2}}{\alpha} )$.
\item[(c)] $\varepsilon \in (\frac{1}{2} ,1]$, $\alpha >0$ and $\varphi\in (0,\frac{\pi}{2\varepsilon})$.
\end{itemize}
Then the problem \eqref{eq.1} has Besov and H\"older maximal regularity in the sense of Theorem \ref{main2} above.
\end{corollary}

\begin{proof}  
This follows directly from Theorem \ref{main2} and the proof of \cite[Theorem 4.1]{ChSr05}.
\end{proof} 

There is an important difference between Corollary \ref{epsilon.1} and \cite[Theorem 4.1]{ChSr05}, which is cited in the proof above and which asserts $L^p$-maximal regularity of the problem \eqref{eq.1}. Compared to \cite[Theorem 4.1]{ChSr05}, Corollary \ref{epsilon.1} contains no further assumptions on the Banach space $X$ and the operator $A$. In particular, $X$ need not be a UMD space and $A$ need not have a bounded $RH^\infty$-functional calculus. The above result applies in general Banach spaces and for general sectorial operators.  However, the conclusion of Corollary \ref{epsilon.1} is not $L^p$-maximal regularity, but rather Besov and H\"older maximal regularity.

\section{The initial value problem} \label{sec.3}

In this section we solve the abstract second order Cauchy problem with initial values in certain trace spaces. Before turning to the Cauchy problem, however, we formulate an abstract theorem in the spirit of Theorem \ref{main}. 

\begin{theorem} \label{thm.initial}
Take the assumptions of Theorem \ref{main}, fix $\theta\in (0,1)$, $1\leq p\leq \infty$, and let $L_{\theta ,p}$  be the operator on $\mathsf{D}_D (\theta ,p)$ as defined in Theorem \ref{main}. Let $\hat{D}$ be a closed extension of $D$, and let $\hat{L}_{\theta ,p}$ be the operator on $\mathsf{D}_{\hat{D}} (\theta ,p)$ given by 
\begin{align*}
 \dom{\hat{L}_{\theta ,p}} & := \{ x \in \dom{\hat{D}^2} \cap \dom{B\hat{D}} \cap \dom{A} : \hat{D}^2 x, \, B\hat{D} x, \, Ax \in \mathsf{D}_{\hat{D}} (\theta , p) \} ,\\
\hat{L}_{\theta ,p} x & := \hat{D}^2 x + B\hat{D} x + Ax . 
\end{align*}
Then, for every $f\in \mathsf{D}_{\hat{D}} (\theta ,p)$ and every $x_0 \in \dom{\hat{L}_{\theta ,p}}$ satisfying the compatibility condition $f - \hat{L}_{\theta ,p}  x_0 \in\mathsf{D}_D (\theta ,p)$ there exists a unique solution $x\in \dom{\hat{L}_{\theta ,p}}$ of the problem
\begin{equation} \label{eq.initial}
\begin{split}
& \hat{D}^2 x + B\hat{D} x + Ax = f , \\ 
& x - x_0 \in \dom{L_{\theta,p}} .
\end{split}
\end{equation}
\end{theorem}

\begin{proof}
{\it Uniqueness} follows from the injectivity of the operator $L_{\theta ,p}$ (Theorem \ref{main}): in fact, if $x_1$, $x_2\in\dom{\hat{L}_{\theta ,p}}$ are two solutions of \eqref{eq.initial}, then $\hat{L}_{\theta ,p} (x_1 - x_2)=0$. On the other hand, $x_1-x_2 = (x_1-x_0) - (x_2-x_0) \in \dom{L_{\theta ,p}}$, so that $\hat{L}_{\theta ,p} (x_1-x_2) = L_{\theta ,p} (x_1-x_2)$ since $\hat{L}_{\theta ,p}$ is an extension of $L_{\theta ,p}$. Now the injectivity of $L_{\theta ,p}$ yields $x_1=x_2$. \\
{\it Existence:} Let $g:= \hat{D}^2 x_0 + B\hat{D} x_0 + Ax_0 = \hat{L}_{\theta ,p} x_0 \in \mathsf{D}_{\hat{D}} (\theta ,p)$. By assumption, $f-g\in\mathsf{D}_D (\theta ,p)$. By Theorem \ref{main}, there exists a unique $x_1\in \dom{L_{\theta ,p}}$ such that 
\[
 L_{\theta ,p} x_1 = D^2 x_1 + BD x_1 + Ax_1 = f-g \in \mathsf{D}_{D} (\theta ,p).
\]
Then $x:= x_0+ x_1$ is a desired solution of the problem \eqref{eq.initial}.
\end{proof}

The assumptions of Theorem \ref{thm.initial} are quite general, and in view of an application to the second order Cauchy problem, we shall successively impose further assumptions.  Let $\hat{D}$ be a closed extension of $D$ such that the operator $\hat{D}^2$ with the natural domain is closed, too. Then the domain $\dom{\hat{L}_{\theta ,p}}$ equipped with the norm
\[
 \| x\|_{\dom{\hat{L}_{\theta ,p}}} := \| x\|_{\mathsf{D}_{\hat{D}} (\theta ,p)} + \| \hat{D}^2 x\|_{\mathsf{D}_{\hat{D}} (\theta ,p)} + \| B\hat{D} x\|_{\mathsf{D}_{\hat{D}} (\theta ,p)} + \| Ax\|_{\mathsf{D}_{\hat{D}} (\theta ,p)} ,
\]
becomes a Banach space. The domain $\dom{L_{\theta ,p}}$ is a closed subspace for the induced norm
\[
 \| x\|_{\dom{L_{\theta ,p}}} := \| x\|_{\mathsf{D}_D(\theta ,p)} + \| D^2 x\|_{\mathsf{D}_D(\theta ,p)} + \| BD x\|_{\mathsf{D}_D(\theta ,p)} + \| Ax\|_{\mathsf{D}_D(\theta ,p)}
\]
or for the equivalent graph norm
\[
\| x \|_{\dom{L_{\theta, p}}} := \| x\|_{\mathsf{D}_D(\theta ,p)} + \| L_{\theta ,p} x \|_{\mathsf{D}_D(\theta ,p)} .
\]
Consider now the bounded operators
\begin{align*}
S_{\theta ,p} : \dom{\hat{L}_{\theta ,p}} & \to \mathsf{D}_{\hat{D}} (\theta ,p) \times \dom{\hat{L}_{\theta ,p}} / \dom{L_{\theta ,p}} \\
 u & \mapsto  ( \hat{L}_{\theta ,p} u , [u]) 
\end{align*}
and
\begin{align*}
T_{\theta ,p} : \mathsf{D}_{\hat{D}} (\theta ,p) \times \dom{\hat{L}_{\theta ,p}} / \dom{L_{\theta ,p}} & \to \mathsf{D}_{\hat{D}} (\theta ,p) / \mathsf{D}_{D} (\theta ,p) , \\
 (f, [u_0] ) & \mapsto [ f - \hat{L}_{\theta ,p} u_0 ] \, , 
\end{align*}
where $u\mapsto [u]$ denotes various quotient maps. Note that $T_{\theta ,p}$ is well defined in the sense that the definition does not depend on the choice of the representative $u_0$. With these definitions one easily sees that the kernel ${\rm ker}\, T_{\theta ,p}$ is exactly the (closed) space of all pairs $(f, [u_0] )$ satisfying the compatibility condition from Theorem \ref{thm.initial} (which does not depend on the representative $u_0$), and that the compatibility condition is necessary for the existence of a solution of \eqref{eq.initial} since $S_{\theta ,p}$ maps into ${\rm ker}\, T_{\theta ,p}$. Theorem \ref{thm.initial} implies that $S_{\theta ,p}$ is an isomorphism onto ${\rm ker}\, T_{\theta ,p}$. 

The drawback of this abstract situation is, however, that in general we have no general description of either the kernel of $T_{\theta ,p}$, or the quotient space $\dom{\hat{L}_{\theta ,p}} / \dom{L_{\theta ,p}}$. 

\begin{example}[The second order Cauchy problem]
Let $A$ and $B$ be two closed, linear operators on a Banach space $X$, and let $1\leq p< \infty$, $1\leq q\leq \infty$ and $\theta\in (0,\frac{1}{p})$. On the space $L^p (0,1;X)$, let the differentiation operators $D$ and $\hat{D}:= D_{max}$ be given as in Example \ref{d.1}. Then $\hat{D}^2$ is closed as one easily verifies.  Unlike in the proof of Theorem \ref{main2}, we denote the multiplication operators on $L^p (0,T;X)$ again by $A$ and $B$, respectively. 

Recall from Example \ref{d.1} that  
\[
 \mathsf{D}_{\hat{D}} (\theta ,p) = \mathsf{D}_{D} (\theta ,p) = B^{\theta}_{pq} (0,T;X) .
\]
Hence,
\[
 \dom{\hat{L}_{\theta ,q}} = \{ u\in W^{2,p} (0,T;X) : \ddot u , \, B\dot u , \, Au \in B^{\theta}_{pq} (0,T;X) \}
\]
and
\[
 \dom{L_{\theta ,q}} = \{ u\in \dom{\hat{L}_{\theta ,q}} : u(0) = \dot u(0) = 0\} .
\]
Note that in this situation, the quotient $\dom{\hat{L}_{\theta ,q}} / \dom{L_{\theta ,q}}$ can be naturally identified with the {\em trace space}
\[
 (X,\dom{B} , \dom{A})_{B^{\theta}_{pq}} := \{ (u_0 ,u_1) \in X\times X : \exists u\in \dom{\hat{L}_{\theta ,q}} \text{ s.t. } u(0) = u_0 , \, \dot u(0) = u_1 \} ,
\]
and the quotient map is then the natural {\em trace operator} $u \mapsto (u(0),\dot u(0))$. For the trace space we use a notation which is similar to the notation of classical real interpolation spaces between a {\em pair} of Banach spaces.  This is appropriate because the classical real interpolation spaces can be identified with trace spaces involving weighted $L^p$ spaces. However, we point out that here we ``interpolate'' between three Banach spaces and that the trace space is a subspace of the product space $X\times X$. 

\begin{corollary}
Let $A$ and $B$ be two closed, linear operators on a Banach space $X$ satisfying the hypotheses (b), (c) and (d) of Theorem \ref{main2}. Let $1\leq p< \infty$, $1\leq q\leq \infty$ and $\theta\in (0,\frac{1}{p})$. Then for every $f\in B^{\theta}_{pq} (0,T;X)$ and every $(u_0,u_1)\in (X,\dom{B},\dom{A})_{B^{\theta}_{pq}}$ the second order Cauchy problem
\begin{equation} \label{second.initial}
\begin{split}
& \ddot u + B \dot u + A u = f  \quad \text{in } [0,T] , \\
& u(0) = u_0 , \quad \dot u (0) = u_1 ,
\end{split} 
\end{equation} 
admits a unique solution $u\in B^{\theta}_{pq} (0,T ;X)$ satisfying
\[
 \dot u, \, \ddot u , \, B\dot u , \, Au \in B^{\theta}_{pq} (0,T;X) .
\]
\end{corollary}

\begin{proof}
Note that for the particular choice of $p$, $q$ and $\theta$ we have $T_{\theta ,q} = 0$, and hence the compatibility condition from Theorem \ref{thm.initial} is empty. In other words, by Theorem \ref{thm.initial}, the operator
\begin{align*}
 S_{\theta ,q} : \dom{\hat{L}_{\theta ,q}} & \to B^{\theta}_{pq} (0,T;X) \times (X,\dom{B} ,\dom{A})_{\theta ,q}  \\
 u & \mapsto  ( \hat{L}_{\theta ,q} u , u(0) ,\dot u(0) )
\end{align*}
is invertible, and this implies the claim.
\end{proof}

We point out that in the particular case $p=1$ there is no restriction on the value of $\theta\in (0,1)$. 

The identification of the trace space $(X,\dom{B},\dom{A})_{B^{\theta}_{pq}}$, even for particular choices of $X$, $B$ and $A$, is left as an open problem.
\end{example}

\section{Examples} \label{sec.4}

\begin{example}[Strong damping] \label{ex.1}
Let $\Omega\subseteq\R^n$ be an open set, and let $\alpha >0$. We consider the following initial-boundary value problem:
\begin{equation} \label{eq.ex.1}
\begin{array}{ll}
u_{tt} - \alpha \Delta u_t -\Delta u = f & \text{in } (0,T)\times\Omega , \\[2mm]
u = 0 & \text{in } (0,T)\times\partial\Omega , \\[2mm]
u(0,x)= 0 & \text{in } \Omega , \\[2mm]
u_t (0,x) = 0 & \text{in } \Omega . 
\end{array}
\end{equation}
For $1\leq r\leq \infty$, we consider the space
\[
 X_r := \begin{cases}
         L^r (\Omega ) & \text{if } 1\leq r <\infty , \\[2mm]
         C_0 (\Omega ) & \text{if } r=\infty .
        \end{cases}
\]
On $X_2 = L^2 (\Omega )$ we consider the negative Dirichlet-Laplace operator $B_2$ given by
\begin{align*}
 \dom{B_2} & := \{ u\in H^1_0 (\Omega ) : \exists f\in L^2 (\Omega ) \, \forall v\in H^1_0 (\Omega ) : \int_\Omega \nabla u \overline{\nabla v} = \int_\Omega f\bar{v} \} , \\
B_2 u & := f .
\end{align*}
It is known that $B_2$ is selfadjoint, nonnegative, and thus sectorial of angle $\varphi = 0$. Hence, the operator $-B_2$ generates an analytic $C_0$-semigroup which is known to have Gaussian upper bounds \cite{ArEl97}, \cite{Da00}, \cite{Ou04}. Thus, if $1\leq r<\infty$, the operator $B_2$, restricted to $X_r \cap L^2 (\Omega )$, extends consistently to a sectorial operator $B_r$ on $X_r$ of angle $\varphi = 0$ \cite[Theorem 2.3]{Hi96}. For domains with uniform $C^2$-boundary, and if $1\leq r < \infty$, one may also refer to \cite[Theorem 3.1.3]{Lu95}, where one finds also the characterization of the domain 
\[
\dom{B_r} = W^{2,r} (\Omega ) \cap W^{1,r}_0 (\Omega ) \text{ if } 1<r<\infty . 
\]
However, we are particularly interested in the end points $r=1$ and $r=\infty$. 

If $1\leq r<\infty$, and if we put $A := B:= B_r$ and $\varepsilon = 1$, then we see that this example is a special case of Corollary \ref{epsilon.1}. We thus obtain the following result.

\begin{corollary} \label{ex.1.cor.1}
Fix $\theta\in (0,1)$ $1 \leq p$, $q\leq \infty$, and $1\leq r<\infty$. Then for every $f\in B^{\theta}_{pq} (0,T; L^r (\Omega ))$ the problem \eqref{eq.ex.1} admits a unique strong solution
\[
 u\in B^{\theta +1}_{pq} (0,T ; \dom{B_r} ) \cap B^{2+\theta}_{pq} (0,T; L^r (\Omega )) .
\]
\end{corollary}

On the space $X_\infty = C_0 (\Omega )$ we take the following realization of the negative Dirichlet-Laplace operator:
\begin{align*}
\dom{B_\infty} & := \{ u\in C_0 (\Omega ) : \Delta u \in C_0 (\Omega ) \} ,\\
 B_\infty u & := - \Delta u .
\end{align*}
It has been shown in \cite[Theorem 1.1]{ArBe99} that $-B_\infty$ is the generator of an analytic semigroup if and only if $\Omega$ is Wiener regular, that is, if and only at each point $x\in\partial\Omega$ there exists a barrier \cite[Definition 3.1]{ArBe99}. A bounded open set $\Omega$ is Wiener regular if and only if the Dirichlet problem
\begin{equation*} \label{eq.wiener}
\begin{array}{ll}
- \Delta u = 0 & \text{in } \Omega , \\[2mm]
u = g & \text{in } \partial\Omega ,
\end{array}
\end{equation*}
admits for each $g\in C (\partial\Omega )$ a unique solution $u\in C (\bar{\Omega})$. Note that in $\R^2$ every bounded, simply connected domain is Wiener regular \cite[Corollary 4.18, p.276]{Co78}. 

If $\Omega$ is Wiener regular, then the operator $B_\infty$ is sectorial of angle $<\frac{\pi}{2}$. Again, if we put $A:= B := B_\infty$ and $\varepsilon = 1$, then we see that Corollary \ref{epsilon.1} applies to problem \eqref{eq.ex.1} and we obtain the following corollary. 

\begin{corollary} \label{ex.1.cor.2}
Assume that $\Omega$ is open and Wiener regular, and fix $\theta\in (0,1)$. Then for every $f\in C^\theta ([0,T] ; C_0 (\Omega ))$ the problem \eqref{eq.ex.1} admits a unique strong solution
\[
 u\in C^{1,\theta} ([0,T];\dom{B_\infty} ) \cap C^{2,\theta} ([0,T] ; C_0 (\Omega )) .
\]
\end{corollary}

\begin{remark}
Note again that the above maximal regularity results apply in particular in the spaces $L^1 (\Omega )$ (Corollary \ref{ex.1.cor.1}) and $C_0 (\Omega )$ (Corollary \ref{ex.1.cor.2}) which are not UMD spaces. Moreover, in Corollary \ref{ex.1.cor.1}, the time regularity allows us to consider also the space $B^\theta_{1,q}$ and in particular $B^\theta_{1,1}$. 
\end{remark}

\end{example}

\begin{example}[Strong damping] \label{ex.1a}
Let $\Omega\subseteq\R^n$ be an open set. We consider now the following initial-boundary value problem:
\begin{equation} \label{eq.ex.2}
\begin{array}{ll}
u_{tt} - {\mathcal A} (x,D) u_t - {\mathcal A} (x,D) u = f & \text{in } (0,T)\times\Omega , \\[2mm]
u = 0 & \text{in } (0,T)\times\partial\Omega , \\[2mm]
u(0,x)= 0 & \text{in } \Omega , \\[2mm]
u_t (0,x) = 0 (x) & \text{in } \Omega . 
\end{array}
\end{equation}
Here ${\mathcal A} (x,D)$ is formally given by
\[
 {\mathcal A} (x,D) u = \sum_{i,j=1}^n D_i (a_{ij} D_j u) + \sum_{i=1}^n (D_i (b_iu) + c_i D_i u) + du 
\]
with real coefficients $a_{ij}$, $b_i$, $c_i$, $d\in L^\infty (\Omega )$ satisfying the ellipticity condition
\[
 \sum_{i,j=1}^n a_{ij} (x) \xi_i \xi_j \geq \eta \, |\xi|^2 
\]
for some $\eta >0$ and all $x\in\Omega$, $\xi\in\R^n$, and the dissipativity condition
\[
 \sum_{i=1}^n D_i b_i + d \leq 0 \text{ in } {\mathcal D} (\Omega )' .
\]
Under these assumptions, we have an operator ${\mathcal A} : H^1 (\Omega ) \to H^{-1} (\Omega )$ given by
\[
 \langle {\mathcal A} u , v \rangle_{H^{-1},H^1_0} := \sum_{i,j=1}^n \int_\Omega a_{ij} D_j u \overline{D_i v} + \sum_{i=1}^n \int_\Omega (b_i u \overline{D_i v} - c_i D_i u \bar{v}) - \int_\Omega du\bar{v} .
\]

We consider the same scale of spaces as in Example \ref{ex.1}. We now define an operator $B_2$ on $X_2 = L^2 (\Omega )$ by
\begin{align*}
\dom{B_2} & := \{ u\in H^1_0 (\Omega ) : \exists f\in L^2 (\Omega ) \, \forall v\in H^1_0 (\Omega ) : \langle {\mathcal A} u , v\rangle_{H^{-1},H^1_0} = \int_\Omega f\bar{v} \} ,  \\  
B_2 u & := f .
\end{align*}
The operator $B_2$ is associated with an elliptic form, it is sectorial of angle $\varphi <\frac{\pi}{2}$, and hence $-B_2$ generates an analytic $C_0$-semigroup. Again, this semigroup has Gaussian upper bounds \cite{ArEl97}, \cite{Da00}, \cite{Ou04}, and if $1\leq r<\infty$, then the operator $B_2$, restricted to $X_r \cap L^2 (\Omega )$, extends consistently to a sectorial operator $B_r$ on $X_r$ of the same angle $\varphi$ \cite[Theorem 2.3]{Hi96}. 
In particular, if $1\leq r<\infty$, and if we put $A := B := B_r$ and $\varepsilon = 1$, then we see that this example is also a special case of Corollary \ref{epsilon.1}. We thus obtain the following result.

\begin{corollary} \label{ex.2.cor.1}
Fix $\theta\in (0,1)$ $1 \leq p$, $q\leq \infty$, and $1\leq r<\infty$. Then for every $f\in B^{\theta}_{pq} (0,T; L^r (\Omega ))$ the problem \eqref{eq.ex.2} admits a unique strong solution
\[
 u\in B^{\theta +1}_{pq} (0,T ; \dom{B_r} ) \cap B^{2+\theta}_{pq} (0,T; L^r (\Omega )) .
\]
\end{corollary}

On the space $X_\infty = C_0 (\Omega )$ we consider the following operator:
\begin{align*}
\dom{B_\infty} & := \{ u\in C_0 (\Omega ) \in H^1_{loc} (\Omega ): {\mathcal A} (x,D) u \in C_0 (\Omega ) \} ,\\
 B_\infty u & := - {\mathcal A} (x,D) u .
\end{align*}
It has been shown in \cite[Corollary 4.7]{ArBe99} that if $\Omega$ is bounded and Wiener regular, then $-B_\infty$ is the generator of an analytic semigroup. Hence, if $\Omega$ is bounded and Wiener regular, then the operator $B_\infty$ is sectorial of angle $<\frac{\pi}{2}$. Again, if we put $A:= B:= B_\infty$ and $\varepsilon = 1$, then we see that Corollary \ref{epsilon.1} applies to problem \eqref{eq.ex.2} and we obtain the following corollary. 

\begin{corollary} \label{ex.2.cor.2}
Assume that $\Omega$ is open, bounded and Wiener regular, and fix $\theta\in (0,1)$. Then for every $f\in C^\theta ([0,T] ; C_0 (\Omega ))$ the problem \eqref{eq.ex.2} admits a unique strong solution
\[
 u\in C^{1,\theta} ([0,T];\dom{B_\infty} ) \cap C^{2,\theta} ([0,T] ; C_0 (\Omega )) .
\]
\end{corollary}

\end{example}

\begin{example}[Intermediate damping] \label{ex.2}
Let $\Omega\subseteq\R^n$ be an open set. We consider the following initial-boundary value problem:
\begin{equation} \label{eq.ex.3}
\begin{array}{ll}
u_{tt} - \alpha \Delta u_t +\Delta^2 u = f & \text{in } (0,T)\times\Omega , \\[2mm]
u = \Delta u = 0 & \text{in } (0,T)\times\partial\Omega , \\[2mm]
u(0,x)= 0 & \text{in } \Omega , \\[2mm]
u_t (0,x) = 0 & \text{in } \Omega . 
\end{array}
\end{equation}
This problem is in fact a special case of the problem \eqref{eq.1} from Corollary \ref{epsilon.1} if we let $1\leq r\leq \infty$, $B_r$ be the negative Dirichlet-Laplace operator on $X_r$ (see Example \ref{ex.1}), and if we put $A=B_r^2$ and $\varepsilon = \frac12$. Then $A$ is still sectorial with angle $\varphi =0$ if $1\leq r<\infty$ and $\varphi\in (0,\pi )$ if $r=\infty$. Moreover, $B_r = A^\frac12$, and we obtain the following two corollaries.

\begin{corollary} \label{ex.3.cor.1}
Fix $\theta\in (0,1)$ $1 \leq p$, $q\leq \infty$, and $1\leq r<\infty$. Assume that $\alpha >0$. Then for every $f\in B^{\theta}_{pq} (0,T; L^r (\Omega ))$ the problem \eqref{eq.ex.3} admits a unique strong solution
\[
 u\in B^\theta_{pq} (0,T; \dom{B_r^2} )\cap B^{1+\theta }_{pq} (0,T ; \dom{B_r} ) \cap B^{2+\theta}_{pq} (0,T; L^r (\Omega )) .
\]
\end{corollary}

\begin{corollary} \label{ex.3.cor.2}
Assume $\alpha \geq 2$, that $\Omega$ is open and Wiener regular, and fix $\theta\in (0,1)$. Then for every $f\in C^\theta ([0,T] ; C_0 (\Omega ))$ the problem \eqref{eq.ex.3} admits a unique strong solution
\[
 u\in C^{\theta} ([0,T];\dom{B_\infty^2} ) \cap C^{1,\theta} ([0,T];\dom{B_\infty} ) \cap C^{2,\theta} ([0,T] ; C_0 (\Omega )) .
\]
\end{corollary}

\end{example}

\def\cprime{$'$} 

\providecommand{\bysame}{\leavevmode\hbox to3em{\hrulefill}\thinspace}
\providecommand{\MR}{\relax\ifhmode\unskip\space\fi MR }
\providecommand{\MRhref}[2]{%
  \href{http://www.ams.org/mathscinet-getitem?mr=#1}{#2}
}
\providecommand{\href}[2]{#2}

\vskip 4 mm

\end{document}